\documentclass[english, 9pt]{amsart}
\usepackage{amssymb}
\usepackage{amsthm}
\usepackage{amsfonts}
\usepackage{amscd}
\usepackage{mathrsfs}
\usepackage{blindtext}
\usepackage{geometry}
\usepackage[T1]{fontenc}
\usepackage{enumerate}%
\usepackage{amsthm,lipsum}
\usepackage{tikz}
\usepackage[colorlinks, linkcolor=red]{hyperref}
\usepackage{color}
\usepackage[utf8]{inputenc}
\usepackage{dsfont}
\definecolor{immi}{rgb}{0,.6,.1}

\newbox\removebox
\newcommand\remove[2]{%
\setbox\removebox=\ifmmode\hbox{$#2$}\else\hbox{#2}\fi%
\leavevmode
\rlap{\textcolor{#1}{\vrule height0.8ex depth-0.6ex width\wd\removebox}}%
\box\removebox
}
\long\def\bigremove#1{%
\par\setbox\removebox=\vbox{#1}%
\vbox{%
\vbox to0pt{\hbox{\tikz\draw[color=blue,thick] (0,0) -- (\wd\removebox,-\ht\removebox)  (\wd\removebox,0) -- (0,-\ht\removebox);}}
\box\removebox
}
}

\usepackage{mathrsfs} %for the \mathscr
\usepackage{bbm}

\makeatletter
\def\RFss@@#1{\RF^*_{\!*#1}}
\def\RFss@_#1{\RFss@@{,#1}}
\def\RFss{\@ifnextchar_{\RFss@}{\RFss@@{}}}
\makeatother

\newcommand{\RF}{{\rm RF}}

\def\11{{\mathbf 1}}

\newcommand{\Q}{{\mathbb Q}}

\newtheorem{thm}[subsection]{Theorem}
\newtheorem{lem}[subsection]
{Lemma}
\newtheorem{cor}[subsection]%[thm]
{Corollary}
\newtheorem{prop}[subsection]
{Proposition}
{Conjecture}
{Problem}
\theoremstyle{plain}
\newtheorem*{namedthm}{\namedthmname}
\newcounter{namedthm}


\theoremstyle{definition}

\theoremstyle{remark}
{Remark}
\newtheorem{rem}[subsection]
{Remark}
{Remarks}

\theoremstyle{plain}

  {\par\medskip\noindent #1\par\begingroup%
    \advance\leftskip by 1em\advance\rightskip by 1em}%
  {\par\endgroup}

\renewcommand{\phi}{\varphi}
\renewcommand{\epsilon}{\varepsilon}
\renewcommand{\theta}{\vartheta}

\def\Q{{\mathbf Q}}

\renewcommand{\and}{ \quad \text{and} \quad }

\begin{document}

\setcounter{tocdepth}{1} % Show subsection in table of contents

\author[T.~Thu~Nguyen]
{Thi Thu Nguyen}
 \address{Université de Lille, Laboratoire Paul Painlevé, 
59655 Villeneuve d'Ascq Cedex, France}
\email{thuclcsphn@gmail.com}
\email{thithu.nguyen@univ-lille.fr}
\subjclass[2020]{Primary 11P32; Secondary 11P55}
\keywords{Goldbach numbers, congruences, Dirichlet L-function, Generalized Riemann Hypothesis, Siegel zero}

\begin{abstract} We obtain asymptotic results on the average numbers of Goldbach representations of an integer as the sum of two primes in  arithmetic progressions.  We also prove an omega-result  showing that the asymptotic result is essentially the best possible.
\end{abstract}

\title[Goldbach in Arithmetic Progressions]{Average orders of Goldbach Estimates in Arithmetic Progressions} 

\maketitle
%\tableofcontents

\section{Introduction}\label{sec:intro}

%\section{}
The Goldbach conjecture that every even integer greater than two can be written as the sum of two primes, is one of the oldest open problems today. Instead of studying directly the Goldbach counting function $g(n)=\sum_{p_1+p_2=n}1$ where $p_1, p_2$ are primes, we consider the corresponding problem for a smoother version using logarithms, that is the weighted Goldbach function
$$G(n) =\sum_{m+l=n}\Lambda(m)\Lambda(l),$$
where $\Lambda(n)$ is the von Mangoldt function. We are interested in the behavior of the average order of magnitude of $G(n)$ for $n\in [1,X]$ where $X$ is a large real number, that is $$S(X):= \sum_{n\leq X} G(n).$$ Fujii  showed the following estimate for $S(X)$.

\begin{thm}\cite[Fujii]{Fujii}
Suppose that the Riemann Hypothesis (RH) is true. Then we have
$$S(X) = \dfrac{X^2}{2}+H(X)+ E(X),$$
where $H(X)=-2\sum_\rho \dfrac{X^{\rho+1}}{\rho(\rho+1)}$ is the sum  over the non-trivial zeros $\rho$ of the Riemann zeta function counted with multiplicity and $E(X) \ll (X\log X)^{4/3}$.
\end{thm}
\begin{rem}
Assuming the RH, the order of magnitude of $H(X)$ is $\mathcal{O}(X^{3/2})$, while unconditionally, it is
$$H(X) \ll X^{1+B},$$ where $B=\sup \{Re(\rho): \zeta(\rho)=0 \}.$
\end{rem}

Granville recovered Fujii's result without assuming RH,  $E(X) \ll X^{\frac{2+4B}{3}}\log ^2X$ \cite{granville1,granville2}. Egami and Matsumoto expected the error term to be  $\mathcal{O}(X^{1+\epsilon})$ (\cite[Conjecture 2.2]{Egami-Mat}). In fact, this 
conjecture was proved by Bhowmik and Schlage-Puchta \cite{BP} who showed  in 2010 that the error term is $\mathcal{O} (X\log^5 X)$ 
under RH and $\Omega(X\log \log X)$ unconditionally. It was improved by Languasco-Zaccagnini \cite{Lang-Zacca1} 
  in 2012 and  using different methods later by  Goldson-Yang \cite{GY}  in 2017 or Goldston-Suriajaya \cite{GS}  in 2023 to  $E(X) \ll X\log^3X.$ 

We study the case of two primes  in arithmetic progressions, which  was first introduced by Rüppel \cite{ruppel2012} in 2012. 
Suzuki \cite{YS} in 2017 obtained an asymptotic formula under the Generalized Riemann Hypothesis with real zeros while for the case of common modulus,  Bhowmik-Halupczok-Matsumoto-Suzuki \cite{BHMS} proved an asymptotic formula without assuming Generalized Riemann Hypothesis (GRH). 

In this paper, we will show an asymptotic  result for the case of different modulus without the GRH. This results in the factor $\log^3(X)$ in the error term thus we can recover Suzuki's result \cite{YS} and improve the upper bound of the error term in \cite[Theorem 2]{BHMS} by a factor $\log^2X$. To do that, we use the properties of the non-trivial zeros of Dirichlet $L$-function to prove a good estimate unconditionally of the second moments of Chebyshev function in arithmetic progressions (see Lemmas \ref{lemh} and \ref{lemk} in section $3$ for more details). Moreover, similar to the classical problem,  we obtain an omega-result in the arithmetic progressions showing that the asymptotic formula is essentially the best possible.

Let  $q_1, q_2$ be  positive integers and  $1 \leq a_1<q_1, 1\leq  a_2 <q_2 $ be positive integers such that $(a_1, q_1)=1$, $(a_2, q_2)=1$. We consider the function 
$$G(n,q_1,q_2,a_1,a_2):= \sum_{\substack{m+l=n\\m\equiv a_1(q_1)\\l\equiv a_2(q_2)}}\Lambda(m)\Lambda(l),$$ whose summatory function is defined as 
$$S(X,q_1,q_2,a_1,a_2)=\sum_{n \leq X} G(n,q_1,q_2,a_1,a_2).$$

We denote by $\chi$  a Dirichlet character modulo $q$. Let $\rho_\chi$ be the non-trivial zeros of $L(s,\chi)$ and let  $B_q=\sup \{Re({\rho_\chi}) \mid L(\rho_\chi, \chi)=0 \}$. We denote  by $\widetilde{\beta}$ the possible exceptional zero (or Siegel zero), $\widetilde{\chi}$ the corresponding exceptional real character  modulo $q$   and $B_q^\sharp =\sup \{Re({\rho_\chi}) \mid L(\rho_\chi, \chi)=0, \rho_\chi \neq \widetilde{\beta} \}$.

In this article, we prove an asymptotic formula for $S(X,q_1,q_2,a_1,a_2)$ and an omega-result, which can be stated as follows.
\begin{thm}\label{t1}
\begin{enumerate}
\item[(A)]
For any $\epsilon>0$, we have
\begin{align}
\begin{split}
S(X,q_1,q_2,a_1,a_2)=\dfrac{X^2}{2\varphi(q_1)\varphi(q_2)}+\dfrac{1}{\varphi(q_2)}H(X,q_1,a_1)+\dfrac{1}{\varphi(q_1)}H(X,q_2,a_2)+\mathcal{Z}(X, \widetilde{\beta}_1, \widetilde{\beta}_2)\\
+\mathcal{O}(X^{B^*_{q_1}+B^*_{q_2}}\log X \log(q_1X) \log (q_2X)),
\end{split}
\end{align}
where the implicit constant  is absolute  and 
\begin{align*}
    &H(X,q,a)=- \dfrac{1}{\varphi(q)}\sum_{\chi(q)}\overline{\chi}(a) \sum_{\rho_\chi} \dfrac{X^{{\rho_\chi}+1}}{{\rho_\chi}({\rho_\chi}+1)},\\
    &\mathcal{Z}(X, \widetilde{\beta}_1, \widetilde{\beta}_2)=\sum_{\substack{\chi_1(q_1)\\\chi_2(q_2)}}\dfrac{\overline{\chi}_1(a_1)\overline{\chi}_2(a_2)}{\varphi(q_1)\varphi(q_2)}  \sum_{ (\rho_{\chi_1},\rho_{\chi_2})\in \mathcal{Z}_0}\dfrac{\Gamma(\rho_{\chi_1})\Gamma(\rho_{\chi_2})}{\Gamma(\rho_{\chi_1}+\rho_{\chi_2}+1)}X^{\rho_{\chi_1}+\rho_{\chi_2}}, 
\end{align*}
with $$\mathcal{Z}_0=\{(\rho_{\chi_1},\rho_{\chi_2}): \rho_{\chi_i} \text{ is a non-trivial zero of  } L(s, \chi_i) \text{ and at least one of } \rho_{\chi_1},\rho_{\chi_2} \text{is a Siegel zero} \},$$
$$B^*_q=B^*_q(X)=\min (B_q^\sharp, 1-\eta), \hspace{3cm} \eta=\eta_q(X)=\dfrac{c_1(\epsilon)}{\log q+ (\log X)^{2/3} (\log \log X)^{1/3}}$$,\\ with some small constant $c_1(\epsilon) >0$.

\item[(B)] We have
\begin{align}
\begin{split}
S(X,q_1,q_2,a_1,a_2)=\dfrac{X^2}{2\varphi(q_1)\varphi(q_2)}+\dfrac{1}{\varphi(q_2)}H(X,q_1,a_1)+\dfrac{1}{\varphi(q_1)}H(X,q_2,a_2)+\mathcal{Z}(X, \widetilde{\beta}_1, \widetilde{\beta}_2)\\
+  \Omega(X\log \log X),
\end{split}
\end{align}
where the implicit constant is depends on $q_1,q_2$.
\end{enumerate}
\end{thm}

%\begin{rem}
%Using Siegel's theorem and the Vinogradov-Korobov zero-free region (see \cite[Theorem 1.1]{zero-free}), we could get a  better result.
%\end{rem}
\begin{rem}
For the unconditional order of magnitude of $H(X,q,a)$, we have 
\begin{align*}
H(X,q,a)&=- \dfrac{1}{\varphi(q)}\sum_{\chi(q)}\overline{\chi}(a) \sum_{\rho_\chi} \dfrac{X^{{\rho_\chi}+1}}{{\rho_\chi}({\rho_\chi}+1)}\\
& \ll \dfrac{X^{1+B_q}}{\varphi(q)}\sum_{\chi(q)}\sum_{\rho_\chi} \dfrac{1}{|{\rho_\chi}({\rho_\chi}+1)|}.
\end{align*}
By using \cite[Lemma 3]{BHMS}, we obtain 
\begin{align*}
H(X,q,a) &\ll \dfrac{X^{1+B_q}}{\varphi(q)}\sum_{\chi(q)}\left( \log^2 q+ \delta_1(\chi)q^{1/2} \log ^2 q \right) \ll X^{1+B_q} \log ^2 q,
\end{align*}
where $ \delta_1(\chi)=1$ if $\chi$ is the exceptional character and $ \delta_1(\chi)=0$ otherwise.
Assuming GRH, we can obtain $$H(X,q,a) \ll_q X^{3/2}.$$ 
\end{rem}

\begin{cor}\label{whenGRH}
Assume GRH. Then we have 
\begin{align}
\begin{split}
S(X,q_1,q_2,a_1,a_2)=\dfrac{X^2}{2\varphi(q_1)\varphi(q_2)}+\dfrac{1}{\varphi(q_2)}H(X,q_1,a_1)+\dfrac{1}{\varphi(q_1)}H(X,q_2,a_2)\\
+\mathcal{O}(X\log X \log (q_1X) \log (q_2X)),
\end{split}
\end{align}
where the implicit constant is absolute.
\end{cor}
\begin{rem}
Assuming GRH with real zeros, we obtain the same result as Suzuki  (\cite[Theorem 1.1]{YS}). 
\end{rem}

When $q_1=q_2=q$, we replace $S(X,q_1,q_2,a_1,a_2)$ by $S(X,q,a_1,a_2)$. Then we obtain the following corollary.
\begin{cor} Let $a_1, a_2$ be integers with $(a_1a_2, q)=1$, we have
\begin{align*}
S(X,q,a_1,a_2)&=\dfrac{X^2}{2\varphi^2(q)}-\dfrac{1}{\varphi^2(q)}\sum_{\chi(q)}(\overline{\chi}(a_1)+\overline{\chi}(a_2)) \sum_{\rho_\chi} \dfrac{X^{{\rho_\chi}+1}}{{\rho_\chi}({\rho_\chi}+1)}+ \mathcal{Z}(X, \widetilde{\beta}, \widetilde{\beta})
+\mathcal{O}(X^{2B^*_{q}}\log X\log^2(qX)),
\end{align*}
where the implicit constant is absolute.
\end{cor}
This result improves  \cite[Theorem 2]{BHMS} where the main term is more explicit and the upper bound $\log^5 (qX)$  is improved by a factor $\frac{\log^3(qX)}{\log X}$. 
When we take $q=1$, we recover the best known result for $S(X)$.

\section{Some calculations for the proof of Theorem \ref{t1}}
To prove the part (A) Theorem \ref{t1}, we use the circle method to compute the sum. Now, we put $z= re(\alpha)$, where $r=e^{-1/X}$ and $e(\alpha):=e^{2\pi i \alpha} $, $|\alpha| \leq 1/2$.  Let $q$ be a positive integer, $1 \leq a <q$ be an integer coprime to $q$ and $1\leq y \leq X$, let us consider a power series  as was done by Hardy-Littlewood \cite{HL}.
\begin{align*}
&F_{a,q}(z):=\sum_{\substack{n\geq 1\\n\equiv a(q)}}\Lambda(n)z^n \hspace{1cm }I_q(z):=\dfrac{1}{\varphi(q)}\sum_{n \geq 1}z^n\\
&\widetilde{I_q}(z):=\dfrac{\widetilde{\chi}(a)}{\varphi(q)}\sum_{n \geq 1}n^{\widetilde{\beta}-1}z^n \hspace{1cm }I(y, z):=\sum_{n \leq y}z^n.
\end{align*} 
Our choice is natural because $F_{a,q}(z)$ is a generating function of primes in arithmetic progression with additive property and we expect it to be close to $I_q(z)-\widetilde{I}_q(z)$ by Dirichlet's theorem.

We will show that $S(X,q_1,q_2,a_1,a_2)$ can be expressed in terms of  $F_{a_1,q_1}(z)$, $F_{a_2,q_2}(z)$ in the following proposition.
\begin{prop}\label{t2}
We have 
\begin{align*}
S(X,q_1,q_2,a_1,a_2)=\dfrac{X^2}{2\varphi(q_1)\varphi(q_2)}+\dfrac{1}{\varphi(q_2)}H(X,q_1,a_1)+\dfrac{1}{\varphi(q_1)}H(X,q_2,a_2)+\mathcal{Z}(X, \widetilde{\beta}_1, \widetilde{\beta}_2)\\
+\mathcal{O}(X\log X(\log(2q_1)+\log(2q_2))+E(X,q_1,q_2,a_1,a_2), 
\end{align*}
where 
\begin{align}
E(X,q_1,q_2,a_1,a_2)=\int^1_0 \left(  F_{a_1,q_1}(z)-I_{q_1}(z)+\widetilde{I}_{q_1}(z)\right) \left(  F_{a_2,q_2}(z)-I_{q_2}(z)+\widetilde{I}_{q_2}(z)\right)  I\left(X, \dfrac{1}{z}\right)  d\alpha. 
\end{align}
\end{prop}

We now prove Proposition \ref{t2}.

We first note that
\begin{align}\label{1}
\begin{split}
&\left(  F_{a_1,q_1}(z)-I_{q_1}(z)+\widetilde{I}_{q_1}(z)\right)\left(  F_{a_2,q_2}(z)-I_{q_2}(z)+\widetilde{I}_{q_2}(z)\right)\\
&= F_{a_1,q_1}(z)F_{a_2,q_2}(z) -F_{a_1,q_1}(z)I_{q_2}(z)-F_{a_2,q_2}(z)I_{q_1}(z)+I_{q_1}(z)I_{q_2}(z)\\
&+F_{a_1,q_1}(z)\widetilde{I}_{q_2}(z)+F_{a_2,q_2}(z)\widetilde{I}_{q_1}(z)-I_{q_1}(z)\widetilde{I}_{q_2}(z)-I_{q_2}(z)\widetilde{I}_{q_1}(z)+\widetilde{I}_{q_1}(z)\widetilde{I}_{q_2}(z).
\end{split}
\end{align}
From the definitions of $I_q(z)$,$\widetilde{I}_q(z)$,  and $F_{a,q}(z)$, we calculate
\begin{align*}
&F_{a_1,q_1}(z)F_{a_2,q_2}(z)=\sum_{\substack{m\geq 1\\m\equiv a_1(q_1)}}\Lambda(m)z^m\sum_{\substack{k\geq 1\\k\equiv a_2(q_2)}}\Lambda(k)z^k=\sum_{n\geq 2}G(n,q_1,q_2,a_1,a_2)z^n,\\
&I_{q_1}(z)I_{q_2}(z)=\dfrac{1}{\varphi(q_1)\varphi(q_2)}\sum_{n\geq 2}(n-1)z^n,\\
&\widetilde{I}_{q_1}(z)\widetilde{I}_{q_2}(z)=\dfrac{\widetilde{\chi_1}(a_1)\widetilde{\chi_2}(a_2)}{\varphi(q_1)\varphi(q_2)}\sum_{n\geq 2}\left(\sum_{m+k=n}m^{\widetilde{\beta}_2-1}k^{\widetilde{\beta}_1-1}\right)z^n.
\end{align*}
Moreover, for $(i,j)\in S_0:= \{(1,2), (2,1)\}$ we have
\begin{align*}
&F_{a_i,q_i}(z)I_{q_j}(z)=\dfrac{1}{\varphi(q_j)}\sum_{\substack{m,k \geq 1\\m\equiv a_i(q_i)}}\Lambda(m)z^{m+k}=\dfrac{1}{\varphi(q_j)}\sum_{n\geq 2}\psi(n-1,q_i,a_i)z^n, \\
&F_{a_i,q_i}(z)\widetilde{I}_{q_j}(z)=\dfrac{\widetilde{\chi_j}(a_j)}{\varphi(q_j)}\sum_{\substack{m,k \geq 1\\m\equiv a_i(q_i)}}\Lambda(m)k^{\widetilde{\beta}_j-1}z^{m+k}=\dfrac{\widetilde{\chi_j}(a_j)}{\varphi(q_j)}\sum_{n\geq 2}\left(\sum_{\substack{m \leq n-1\\m\equiv a_i(q_i)}}\Lambda(m)(n-m)^{\widetilde{\beta}_j-1}\right)z^n,\\
&I_{q_i}(z)\widetilde{I}_{q_j}(z)=\dfrac{\widetilde{\chi_j}(a_j)}{\varphi(q_1)\varphi(q_2)}\sum_{\substack{m,k \geq 1}}k^{\widetilde{\beta}_j-1}z^{m+k}=\dfrac{\widetilde{\chi_j}(a_j)}{\varphi(q_1)\varphi(q_2)}\sum_{n\geq 2}\left(\sum_{k\leq n-1}k^{\widetilde{\beta}_j-1}\right)z^n.
\end{align*}
where $\psi$ is the Chebyshev function in arithmetic progression, defined as
 $$\psi(t,q,a)=\sum_{\substack{m\leq t\\m \equiv a(q)}}\Lambda(m).$$
Hence, the right hand side of \eqref{1} can be rewritten as
$$\sum_{n\geq 2} B(n,q_1,q_2,a_1,a_2)z^n ,$$ with
\begin{align}\label{B}
\begin{split}
&B(n,q_1,q_2,a_1,a_2)=G(n,q_1,q_2,a_1,a_2)+\dfrac{n-1}{\varphi(q_1)\varphi(q_2)}+\dfrac{\widetilde{\chi_1}(a_1)\widetilde{\chi_2}(a_2)}{\varphi(q_1)\varphi(q_2)}\left(\sum_{m+k=n}m^{\widetilde{\beta}_2-1}k^{\widetilde{\beta}_1-1}\right)\\
&-\sum_{(i,j) \in S_0}\dfrac{\psi(n-1,q_i,a_i)}{\varphi(q_j)}
+\sum_{(i,j) \in S_0}\dfrac{\widetilde{\chi_j}(a_j)}{\varphi(q_j)}\left(\sum_{\substack{m \leq n-1\\m\equiv a_i(q_i)}}\Lambda(m)(n-m)^{\widetilde{\beta}_j-1}\right)-\sum_{(i,j) \in S_0}\dfrac{\widetilde{\chi_j}(a_j)}{\varphi(q_1)\varphi(q_2)}\left(\sum_{k\leq n-1}k^{\widetilde{\beta}_j-1}\right).
\end{split}
\end{align}
So that
\begin{align*}
E(X, q_1,q_2,a_1,a_2)&= \int_0^1\sum_n\sum_{m\leq X} B(n,q_1,q_2,a_1,a_2)r^n e(n\alpha)r^{-m} e(-m\alpha) d\alpha\\
&=\sum_n\sum_{m\leq X} B(n,q_1,q_2,a_1,a_2)r^{n-m}\int_0^1e((n-m)\alpha) d\alpha\\
&=\sum_{n\leq X} B(n,q_1,q_2,a_1,a_2)
\end{align*}
since $\int_0^1e(t\alpha) d\alpha= 1 \text{ if $t=0$ and }0 \text{ otherwise.} $
Then we obtain the equation 
\begin{align}\label{S}
\begin{split}
S(X,q_1,q_2,a_1,a_2)&=E(X,q_1,q_2,a_1,a_2)-\sum_{n\leq X}\dfrac{n-1}{\varphi(q_1)\varphi(q_2)}-\dfrac{\widetilde{\chi_1}(a_1)\widetilde{\chi_2}(a_2)}{\varphi(q_1)\varphi(q_2)}\sum_{n\leq X}\left(\sum_{m+k=n}m^{\widetilde{\beta}_2-1}k^{\widetilde{\beta}_1-1}\right)\\
 &+\sum_{(i,j) \in S_0}\dfrac{1}{\varphi(q_j)}\sum_{n\leq X}\psi(n-1,q_i,a_i)
-\sum_{(i,j) \in S_0}\dfrac{\widetilde{\chi_j}(a_j)}{\varphi(q_j)}\sum_{n\leq X}\left(\sum_{\substack{m \leq n-1\\m\equiv a_i(q_i)}}\Lambda(m)(n-m)^{\widetilde{\beta}_j-1}\right)\\
&+\sum_{(i,j) \in S_0}\dfrac{\widetilde{\chi_j}(a_j)}{\varphi(q_1)\varphi(q_2)}\sum_{n\leq X}\left(\sum_{k\leq n-1}k^{\widetilde{\beta}_j-1}\right).
\end{split}
\end{align}
We now estimate the sums in equation \eqref{S}. We start from the following lemma.
\begin{lem}\label{calculation} For $0<\beta<1$ and $0< x< y$, we have 
\begin{align}\label{sum-integral}
\sum_{x< n\leq y}n^{\beta-1}=\dfrac{y^\beta-x^\beta}{\beta}+\mathcal{O}(1).
\end{align}
Consequence, for $X\geq 2$ and $1/2 <\beta, \beta_1, \beta_2 <1$, we have
    \begin{align*}
    &\sum_{n \leq X}\left(\sum_{k\leq n-1}k^{\beta-1}\right)=\dfrac{X^{\beta+1}}{\beta(\beta+1)}+\mathcal{O}(X),\\
    &\sum_{n \leq X}\left(\sum_{m+k=n}m^{\beta_2-1}k^{\beta_1-1}\right)=\dfrac{\Gamma(\beta_1)\Gamma(\beta_2)}{\Gamma(\beta_1+\beta_2+1)}X^{\beta_1+\beta_2}+\mathcal{O}(X).
\end{align*}
\end{lem}
\begin{proof}
If $f$ is defined on $[1,x]$ monotonic and continuous, then 
\begin{align*}
    \sum_{n\leq x}f(n) = \int_1^x f(t)dt +\mathcal{O}(f(x)+f(1)).
\end{align*}
Then we obtain, 
\begin{align*}
    \sum_{x< n\leq y}n^{\beta-1}=\int_x^y t^{\beta-1} dt +\mathcal{O}(x^{\beta-1}+y^{\beta-1})=\dfrac{y^\beta-x^\beta}{\beta}+\mathcal{O}(1),
\end{align*}
since $0<\beta <1$.
Using this formula, we obtain
\begin{align*}
    &\sum_{k\leq n-1}k^{\beta-1}=\dfrac{(n-1)^\beta-1}{\beta}+\mathcal{O}(1)=\dfrac{n^\beta}{\beta}+\mathcal{O}(1),\\
    &\sum_{n \leq X}\left(\sum_{k\leq n-1}k^{\beta-1}\right)= \sum_{n \leq X}\left(\dfrac{n^{\beta}}{\beta}+\mathcal{O}(1)\right)=\dfrac{X^{\beta+1}}{\beta(\beta+1)}+\mathcal{O}(X),
\end{align*}
since $1/2 <\beta<1$.

Similarly, we can prove the remaining formula in this lemma.
\begin{align*}
    \sum_{m+k=n}m^{\beta_2-1}k^{\beta_1-1}&=\sum_{k\leq n-1}(n-k)^{\beta_2-1}k^{\beta_1-1}=\int_1^{n-1}(n-t)^{\beta_2-1}t^{\beta_1-1} dt+\mathcal{O}((n-1)^{\beta_1-1}+(n-1)^{\beta_2-1})\\
    &=\int_0^{n}(n-t)^{\beta_2-1}t^{\beta_1-1} dt+\mathcal{O}(1)=n^{\beta_1+\beta_2-1}\int_0^1(1-t)^{\beta_2-1}t^{\beta_1-1} dt+\mathcal{O}(1)\\
    &=n^{\beta_1+\beta_2-1}\dfrac{\Gamma(\beta_1)\Gamma(\beta_2)}{\Gamma(\beta_1+\beta_2)}+\mathcal{O}(1).
\end{align*}
Hence, by \eqref{sum-integral} one has
\begin{align*}
    \sum_{n \leq X}\left(\sum_{m+k=n}m^{\beta_2-1}k^{\beta_1-1}\right)=\dfrac{\Gamma(\beta_1)\Gamma(\beta_2)}{\Gamma(\beta_1+\beta_2+1)}X^{\beta_1+\beta_2}+\mathcal{O}(X).
\end{align*}
\end{proof}

We now study an explicit formula for the average of $\psi(t,q,a)$.
Using the orthogonality relation of Dirichlet characters, we have
\begin{align}\label{ortho_cheby}
    \psi(t,q,a)=\dfrac{1}{\varphi(q)}\sum_{\chi(q)}\overline{\chi}(a)\psi (t,\chi), 
\end{align}
where \begin{align*}
    \psi (t,\chi):=\sum_{n\le t}\chi(n)\Lambda(n).
\end{align*}
\begin{lem}\label{function_chi}
    For a primitive Dirichlet character $\chi$ modulo $q$ and $t,T \geq 2$, we have 
    \begin{align*}
     \psi (t,\chi)=\delta_0(\chi)t-\sum_{{1<\mid \gamma_\chi\mid \leq T}}\dfrac{t^{\rho_\chi}}{\rho_\chi}-\sum_{{\mid \gamma_\chi\mid \leq 1}}\dfrac{t^{\rho_\chi}-1}{\rho_\chi} + R(t,T,\chi)
\end{align*}
with $$R(t,T,\chi)\ll \log (qt)+ \dfrac{t\log^2 (tqT)}{T},$$
 and $ \delta_0(\chi)=1$ if $\chi=\chi_0$, $ \delta_0(\chi)=0$ otherwise.
\end{lem}
\begin{proof}
    If $\chi$ is the trivial character, we have
    $$\sum_{|\rho_\chi|\leq \frac{1}{4}}\dfrac{1}{\rho_\chi}\ll 1 $$and so the assertion follows by \cite[Theorem 12.5]{MV}.

    If $\chi$ is not the trivial character, we first use \cite[Theorem 12.10]{MV}. Then, the
assertion follows once we prove
\begin{align}\label{L function}
    \dfrac{L'}{L}(1, \overline{\chi})=\sum_{|\gamma_\chi|\leq 1}\dfrac{1}{\rho_\chi}+\log q.
\end{align}
By \cite[Lemma 11.1]{MV} and the bound
\begin{align*}
    \sum_{\substack{|\rho_{\overline{\chi}}-\frac{3}{2}|\leq \frac{5}{6}\\ |1-\gamma_{\overline{\chi}}|>1}}\dfrac{1}{1-\rho_{\overline{\chi}}}+\sum_{\substack{|\rho_{\overline{\chi}}-\frac{3}{2}|> \frac{5}{6}\\ |1-\gamma_{\overline{\chi}}|\leq1}}\dfrac{1}{1-\rho_{\overline{\chi}}} \ll \sum_{|\rho_{\overline{\chi}}-\frac{3}{2}|\leq \frac{5}{6}} 1+\sum_{|\gamma_{\overline{\chi}}|\leq 1}1 \ll \log q,
\end{align*}
we obtain 
$$\dfrac{L'}{L}(1, \overline{\chi})=\sum_{ |1-\gamma_{\overline{\chi}}|\leq1}\dfrac{1}{1-\rho_{\overline{\chi}}} +\log q.$$ We can
change variables via $\rho_\chi=1-\rho_{\overline{\chi}} $, which are the zeros of $L(s, \chi) $ by the functional equation of $L(s, \chi)$ and obtain \eqref{L function}.
\end{proof}
\begin{rem}\label{imprimitive}
    If $\chi$ is imprimitive. Suppose that $\chi$ is a character modulo $q$ induced by the primitive character $\chi^*$ modulo $q^*$, $q^*>1$.
Then we have 
\begin{align*}
     \psi (t,\chi)- \psi (t,\chi^*)&\ll \sum_{\substack{n\le t\\(n,q)>1}} \Lambda(n)=\sum_{p|q}\sum_{p^k\le t}\log p\ll \sum_{p|q}\left[\dfrac{\log t}{\log p}\right]\log p\\
     &\ll \log t \sum_{p|q}\log p\ll \log q \log t.
\end{align*}
\end{rem}

\begin{lem} \label{sumphi} For $X \geq 2$ and $(a,q)=1$, we have
\begin{align*}
\sum_{n\leq X}\psi(n-1,q,a)=\dfrac{X^2}{2\varphi(q)}-\dfrac{1}{\varphi(q)}\sum_{\chi(q)}\overline{\chi}(a)\sum_{\rho_\chi}\dfrac{X^{\rho_\chi+1}}{\rho_\chi(\rho_\chi+1)}+\mathcal{O}(X\log (2q)\log (X)).
\end{align*}
\end{lem}
\begin{proof}
We first remark that 
\begin{align}\label{phi1}
\sum_{n\leq X}\psi(n-1,q,a)= \sum_{\substack{n\leq X\\n\equiv a(q)}}(X-n)\Lambda(n)=\int_1^X\psi(t,q,a)dt+\mathcal{O}(\psi(X,q,a)).
\end{align}
Then from Lemma \ref{function_chi} and the formula \eqref{ortho_cheby}, we obtain for $t, T \geq 2$

\begin{align*}
\psi(t,q,a)=\dfrac{t}{\varphi(q)}-\dfrac{1}{\varphi(q)}\sum_{\chi(q)}\overline{\chi}(a)\sum_{{\mid \gamma_\chi\mid \leq T}}\dfrac{t^{\rho_\chi}}{\rho_\chi}+ \dfrac{1}{\varphi(q)}\sum_{\chi(q)}\overline{\chi}(a)\sum_{{\mid \gamma_\chi\mid \leq 1}}\dfrac{1}{\rho_\chi}+\mathcal{O}\left(\log (2q) \log t+\dfrac{t}{T} \log^2 (tqT) \right).
\end{align*} 
This leads to the following
\begin{align}\label{phi2}
\psi(t,q,a)=&\dfrac{t}{\varphi(q)}-\dfrac{1}{\varphi(q)}\sum_{\chi(q)}\overline{\chi}(a)\sum_{{\mid \gamma_\chi\mid \leq T}}\dfrac{t^{\rho_\chi}}{\rho_\chi}+\mathcal{O}\left(\log (2q) \log t+\dfrac{t}{T} \log^2 (tqT) \right),
\end{align}
since there is at most one exceptional character mod $q$.\\
Substituting \eqref{phi2} into \eqref{phi1}, we conclude that
\begin{align*}
\sum_{n\leq X}\psi(n-1,q,a)&=\dfrac{X^2}{2\varphi(q)}-\dfrac{1}{\varphi(q)}\sum_{\chi(q)}\overline{\chi}(a)\sum_{\mid \gamma_\chi\mid \leq T}\dfrac{X^{\rho_\chi+1}}{\rho_\chi(\rho_\chi+1)}+\mathcal{O}\left(X\log (2q) \log X+ \dfrac{X^2}{T}\log^2(qX)\right),
\end{align*}
since $\psi(X,q,a) \ll X$.
Therefore Lemma \ref{sumphi} is proved when we choose $T$ tending to infinity. 
\end{proof}
Now we study the the following sum.
\begin{lem} \label{n-m,beta}
For any $1/2<\beta<1$, we have
    \begin{align*}
        \sum_{n\leq X}\left(\sum_{\substack{m \leq n-1\\m\equiv a(q)}}\Lambda(m)(n-m)^{\beta-1}\right)=\dfrac{1}{\varphi(q)}&\left(\dfrac{X^{\beta+1}}{\beta(\beta+1)}-\sum_{\chi(q)}\overline{\chi}(a)\sum_{\rho_\chi}\dfrac{\Gamma(\beta)\Gamma(\rho_\chi)}{\Gamma(\beta+\rho_\chi+1)}X^{\rho_\chi+\beta}\right)\\
        &+\mathcal{O}\left(X\log (2q) \log X\right).
    \end{align*}
\end{lem}
\begin{proof}
    We have 
    \begin{align}\label{tong}
        \sum_{n\leq X}\left(\sum_{\substack{m \leq n-1\\m\equiv a(q)}}\Lambda(m)(n-m)^{\beta-1}\right)=\dfrac{1}{\varphi(q)}\sum_{\chi(q)}\overline{\chi}(a)\int_1^X\psi(t,\chi)(X-t)^{\beta-1}dt+\mathcal{O}(X).
    \end{align}
Using Lemma \ref{function_chi}, we calculate the right-hand side of \eqref{tong} is equal to, 
\begin{align*}
    &\int_1^X\left(\delta_0(\chi)t-\sum_{{\mid \gamma_\chi\mid \leq T}}\dfrac{t^{\rho_\chi}}{\rho_\chi}+\sum_{|\gamma_\chi|\leq 1}\dfrac{1}{\rho_\chi}\right)(X-t)^{\beta-1}dt+ \mathcal{O}\left(\left(\log (2q) \log X+\dfrac{X\log^2 (XqT)}{T} \right)\int_1^X(X-t)^{\beta-1}dt\right)\\
    &=\delta_0(\chi)\dfrac{X^{\beta+1}}{\beta(\beta+1)}-\sum_{{\mid \gamma_\chi\mid \leq T}}\dfrac{\Gamma(\beta)\Gamma(\rho_\chi)}{\Gamma(\beta+\rho_\chi)}\dfrac{X^{\rho_\chi+\beta}}{\rho_\chi+\beta}+\sum_{|\gamma_\chi|\leq 1}\dfrac{1}{\rho_\chi}\dfrac{X^{\beta}}{\beta}+\mathcal{O}\left(\dfrac{X^\beta}{\beta}\left(\log (2q) \log X+\dfrac{X\log^2 (XqT)}{T} \right)\right).
\end{align*}
Then we can infer the left-hand side of \eqref{tong} is equal to
\begin{align*}
    \dfrac{1}{\varphi(q)}\dfrac{X^{\beta+1}}{\beta(\beta+1)}-\dfrac{1}{\varphi(q)}\sum_{\chi(q)}\overline{\chi}(a)\sum_{{\mid \gamma_\chi\mid \leq T}}\dfrac{\Gamma(\beta)\Gamma(\rho_\chi)}{\Gamma(\beta+\rho_\chi)}\dfrac{X^{\rho_\chi+\beta}}{\rho_\chi+\beta}+\mathcal{O}\left(\dfrac{X^\beta}{\beta}\left(\log (2q) \log X+\dfrac{X\log^2 (XqT)}{T} \right)\right).
\end{align*}
Hence, this Lemma is completed when we choose $T$ tending to infinity.
\end{proof}
The proof of Proposition \ref{t2} is complete when we substitute Lemma \ref{calculation}, Lemma \ref{sumphi} and Lemma \ref{n-m,beta} into \eqref{S}. $\hfill\square$

In the next section, we study the estimate of two important functions for the proof of Theorem \ref{t1}.
\section{Second moments of Chebyshev function in arithmetic progressions}
For $x\geq 2$ and $1\leq h\leq x$, we define
\begin{align*}
    H_{a,q}(x)&:=\int_{0}^x  \left(\psi(t,q,a)-\dfrac{t}{\varphi(q)}+ \dfrac{\widetilde{\chi}(a)}{\varphi(q)}\dfrac{t^{\widetilde{\beta}}}{\widetilde{\beta}}\right)^2 dt,\\
    K_{a,q}(x,h)&:=\int_0^x \left( \psi(t+h,q,a)-\psi(t,q,a)-\dfrac{h}{\varphi(q)}+ \dfrac{\widetilde{\chi}(a)}{\varphi(q)}\dfrac{(t+h)^{\widetilde{\beta}}-t^{\widetilde{\beta}}}{\widetilde{\beta}}\right) ^2dt.
\end{align*} 
First, we prove the following lemmas.
\begin{lem}\cite[Theorem 10.17]{MV}\label{number zeros}  For any $T\geq 0$, we have
    \begin{equation*}
        \sum_{\substack{\rho_\chi\\T\leq |\gamma_\chi|\leq T+1}}1\ll \log q (T+2).
    \end{equation*}
\end{lem}
\begin{lem}\label{2zero}
Let $\rho_\chi=  \beta_\chi+i\gamma_\chi$ be a non-trivial zero of $L(s,\chi)$. Then we have
\begin{align}\label{2zero1}
\sum_{ \mid \gamma'_\chi\mid >1}\dfrac{1}{|\gamma'_\chi|(1+|\gamma_\chi-\gamma'_\chi|)}\ll \dfrac{\log (q|\gamma_\chi|)\log (|\gamma_\chi|)}{|\gamma_\chi|},
\end{align}
\begin{align}\label{2zero2}
\sum_{\gamma'_\chi} \dfrac{1}{(1+|\gamma_\chi-\gamma'_\chi|)^2} \ll \log (q|\gamma_\chi|),
\end{align}
where $\sum_{ \gamma'_\chi}$ denotes a sum over the non-trivial zeros of Dirichlet $L$-function associated to $\chi$ modulo $q$.
\end{lem}
\begin{proof}
Firstly, to prove \eqref{2zero1}, we split into three cases. 
    
\textit{Case 1:} If $|\gamma'_\chi| > 2|\gamma_\chi|$. Then we have 
$$|\gamma_\chi-\gamma'_\chi| \geq |\gamma'_\chi|-|\gamma_\chi| >\dfrac{|\gamma'_\chi|}{2}. $$
Thus, by using Lemma \cite{BHMS}, we obtain 
\begin{align*}
    \sum_{ \substack{\gamma'_\chi\\|\gamma'_\chi| >2|\gamma_\chi|}}\dfrac{1}{|\gamma'_\chi|(1+|\gamma_\chi-\gamma'_\chi|)}\ll  \sum_{ \substack{\gamma'_\chi\\|\gamma'_\chi| > 2|\gamma_\chi|}}\dfrac{1}{|\gamma'_\chi|^2} \ll \dfrac{\log (q|\gamma_\chi|)}{|\gamma_\chi|}.
\end{align*}
\textit{Case 2:} If $|\gamma'_\chi| <\dfrac{1}{2}|\gamma_\chi|$. Then we have 
$$|\gamma_\chi-\gamma'_\chi| \geq |\gamma_\chi|-|\gamma'_\chi| >\dfrac{|\gamma_\chi|}{2}. $$
So we obtain
\begin{align*}
    \sum_{ \substack{\gamma'_\chi\\1<|\gamma'_\chi| <\dfrac{1}{2}|\gamma_\chi|}}\dfrac{1}{|\gamma'_\chi|(1+|\gamma_\chi-\gamma'_\chi|)}\ll \dfrac{1}{|\gamma_\chi|} \sum_{ \substack{\gamma'_\chi\\1<|\gamma'_\chi| < \dfrac{1}{2}|\gamma_\chi|}}\dfrac{1}{|\gamma_\chi'|} \ll \dfrac{\log (q|\gamma_\chi|)}{|\gamma_\chi|}.
\end{align*}
\textit{Case 3:} If $\dfrac{1}{2}|\gamma_\chi|\leq|\gamma'_\chi| \le 2|\gamma_\chi|$, we have 
$$|\gamma_\chi-\gamma'_\chi| \le |\gamma'_\chi|+|\gamma_\chi| \le 3|\gamma_\chi|. $$This implies 
\begin{align*}
    &\sum_{ \substack{\gamma'_\chi\\|\gamma'_\chi| \le 2|\gamma_\chi|}}\dfrac{1}{|\gamma'_\chi|(1+|\gamma_\chi-\gamma'_\chi|)}\le \sum_{ \substack{\gamma'_\chi\\|\gamma_\chi-\gamma'_\chi| \le 3|\gamma_\chi|}}\dfrac{1}{|\gamma'_\chi|(1+|\gamma_\chi-\gamma'_\chi|)}\\
    & =\sum_{ \substack{\gamma'_\chi\\|\gamma_\chi-\gamma'_\chi| \le 1}} \dfrac{1}{|\gamma'_\chi|(1+|\gamma_\chi-\gamma'_\chi|)}+ \sum_{ \substack{\gamma'_\chi\\1<|\gamma_\chi-\gamma'_\chi| \le 3|\gamma_\chi|}} \dfrac{1}{|\gamma'_\chi|(1+|\gamma_\chi-\gamma'_\chi|)}\\
    & \le \sum_{ \substack{\gamma'_\chi\\|\gamma_\chi-\gamma'_\chi| \le 1}} \dfrac{1}{|\gamma_\chi|}+ \sum_{1< n \le 3|\gamma_\chi|}\dfrac{1}{n}\sum_{\substack{\gamma'_\chi\\n<|\gamma_\chi-\gamma'_\chi|\le n+1}} \dfrac{1}{|\gamma'_\chi|}.
\end{align*}
By using Lemma \ref{number zeros}, we can estimate the last sum to be
\begin{align*}
  \ll  \dfrac{\log (q|\gamma_\chi|)}{|\gamma_\chi|} \left(1+ \sum_{1< n \le 3|\gamma_\chi|}\dfrac{1}{n} \right) \ll \dfrac{\log (q|\gamma_\chi|)\log (|\gamma_\chi|)}{|\gamma_\chi|}.
\end{align*}

So we obtain the stated result. We can also prove \eqref{2zero2}. 
\end{proof}

\begin{lem}
\label{lem:sum_low_lying_zero}
For $q\ge1$ and $t\ge1$, we have
\begin{equation*}
\biggl|
\sum_{\substack{
|\gamma_{\chi}|\le 1\\
\rho_{\chi}\neq \widetilde{\beta}
}}\frac{t^{\rho_{\chi}}-1}{\rho_{\chi}}
\biggr|
\ll
t^{B_{q}^{\ast}}\log(2q).
\end{equation*}
\end{lem}
%%%%%%%%%%%%%%%%%%%%%%%%%%%%%%%%%%%%%%%%
\begin{proof}
When $|\rho_{\chi}|\ge\frac{1}{4}$ and $\rho_{\chi}\neq \widetilde{\beta}$, we just have
\[
\frac{t^{\rho_{\chi}}-1}{\rho_{\chi}}
\ll
t^{\beta_{\chi}}+1
\ll
t^{B_{q}^{\ast}}.
\]
When $|\rho_{\chi}|\le\frac{1}{4}$ and so $\beta_{\chi}\le\frac{1}{4}$, we have
\[
\frac{t^{\rho_{\chi}}-1}{\rho_{\chi}}
=
\frac{\log t}{\rho_{\chi}}
\int_{0}^{\rho_{\chi}}t^{s}ds,
\]
where the contour of integration is the line segment from $0$ to $\rho_{\chi}$,
and then
\begin{align*}
\frac{t^{\rho_{\chi}}-1}{\rho_{\chi}}
&\le
(\log t)\sup_{0\le\sigma\le\beta_{\chi}}t^{\beta_{\chi}}
\ll
t^{\beta_{\chi}}(\log t)
\le
t^{\frac{1}{4}}(\log t)
\ll
t^{\frac{1}{2}}
\ll
t^{B_{q}^{\ast}}.
\end{align*}
since $B_{q}^{\ast}\ge\frac{1}{2}$. We thus have
\[
\biggl|
\sum_{\substack{
|\gamma_{\chi}|\le 1\\
\rho_{\chi}\neq \widetilde{\beta}
}}\frac{t^{\rho_{\chi}}-1}{\gamma_{\chi}}
\biggr|
\ll
t^{B_{q}^{\ast}}
\sum_{|\rho_{\chi}|\le 1}1
\ll
t^{B_{q}^{\ast}}\log (2q).
\] This completes the proof.
\end{proof}
Now we show a good upper bound of  the functions $H_{a,q}(x)$ and $K_{a,q}(x,h)$ as follows.
\begin{lem}\label{lemh}
We have the estimate
$$H_{a,q}(x) \ll x^{2B^*_q+1}\log^2(qx).$$
\end{lem}
\begin{proof}

When $q\ge x\ge t$, we have
\begin{align*}
\psi(t,a,q)
-
\frac{t}{\phi(q)}
+
\dfrac{\widetilde{\chi}(a)}{\varphi(q)}\dfrac{t^{\widetilde{\beta}}}{\widetilde{\beta}}
\ll
\log t\sum_{\substack{
n\le t\\
n\equiv a(q)
}}
1
+
\frac{t\log q}{q}
\ll
\log (qt),
\end{align*}
where we used the bound $\phi(q)\gg q/\log q$ \cite[Theorem 2.9]{MV}.
We thus have
\begin{equation*}
\int_{1}^{x}
\biggl(
\psi(t,a,q)
-
\frac{t}{\phi(q)}
+
\dfrac{\widetilde{\chi}(a)}{\varphi(q)}\dfrac{t^{\widetilde{\beta}}}{\widetilde{\beta}}
\biggr)^{2}dt
\ll
x\log ^2(qx)
\ll
x^{2B_{q}^{\ast}+1}
\log^2( qx)
\end{equation*}
if $q\ge x$. So we may assume $q\le x$.

By Lemma \ref{function_chi}, we can rewrite 
\begin{align}\label{H1}
\begin{split}
H_{a,q}(x)&\ll \dfrac{1}{\varphi^2(q)}\int_0^x   \sum_{\chi(q)}|\overline{\chi}(a)|^2 \left( \sum_{\substack{1 < \mid \gamma_\chi\mid \leq T}}\dfrac{t^{\rho_\chi}}{\rho_\chi}+ \sum_{\substack{\mid \gamma_\chi\mid \leq 1\\\rho_\chi \neq\widetilde{ \beta}} }\dfrac{t^{\rho_\chi}-1}{\rho_\chi}\right) ^2dt \\
&\hspace{1cm}+ \int_0^x \left(\log (2q)\log t+\dfrac{t}{T} \log^2 (tqT) \right)^2dt\\
&\ll \dfrac{1}{\varphi^2(q)}\sum_{\chi(q)}|\overline{\chi}(a)|^2 \left( \int_1^x \left|   \sum_{\substack{1 < \mid \gamma_\chi\mid \leq T}}\dfrac{t^{\rho_\chi}}{\rho_\chi}\right|  ^2dt+\int_1^x \left|   \sum_{\substack{\mid \gamma_\chi\mid \leq T\\\rho_\chi \neq \widetilde{ \beta}} }\dfrac{t^{\rho_\chi}-1}{\rho_\chi}\right|  ^2dt\right)\\
&\hspace{1cm}+x\log^2(2q)\log^2x+\dfrac{x^3}{T^2}\log^4(xqT)
\end{split}
\end{align}
since $$\int_0^1 \left|   \sum_{\substack{\mid \gamma_\chi\mid \leq T\\\rho_\chi \neq \widetilde{ \beta}} }\dfrac{t^{\rho_\chi}-1}{\rho_\chi}\right|  ^2dt\ll 1.$$
By Lemma \ref{lem:sum_low_lying_zero}, we have 
\begin{align*}
\int_1^x \left|  \sum_{\substack{\mid \gamma_\chi\mid \leq 1\\\rho_\chi \neq \widetilde{ \beta}} }\dfrac{t^{\rho_\chi}-1}{\rho_\chi}\right|  ^2dt \ll \int_1^x (t^{B_q^*}\log (2q))^2dt\ll x^{2B_q^*+1}\log^2(2q).
\end{align*}
Now we estimate  the remaining sum, that is
\begin{align*}
\int_1^x \left|  \sum_{\substack{1 < \mid \gamma_\chi\mid \leq T}}\dfrac{t^{\rho_\chi}}{\rho_\chi}\right|  ^2dt
 &\ll \sum_{\substack{1 < \mid \gamma_\chi\mid \leq T}}\sum_{\substack{1 < \mid \gamma'_\chi\mid \leq T}}\dfrac{x^{\beta_\chi+\beta'_\chi+1}}{|\gamma_\chi||\gamma'_\chi|(1+|\gamma_\chi-\gamma'_\chi|)}\\
&\ll x^{2B^*_q+1} \sum_{1<\mid \gamma_\chi \mid \leq T}\sum_{1<\mid \gamma'_\chi \mid \leq T}\dfrac{1}{|\gamma_\chi||\gamma'_\chi|(1+|\gamma_\chi-\gamma'_\chi|)}.
\end{align*}
Moreover, using Lemma \ref{2zero}, one has
\begin{align*}
\sum_{1<|\gamma_\chi|  \leq T}\sum_{ \gamma'_\chi }\dfrac{1}{|\gamma_\chi||\gamma'_\chi|(1+|\gamma_\chi-\gamma'_\chi|)} &\ll \sum_{1<|\gamma_\chi|  \leq T} \dfrac{\log (q|\gamma_\chi|)\log (|\gamma_\chi|)}{|\gamma_\chi|^2}\\
&\ll \log^2(2q).
\end{align*}
So
$$H_{a,q}(x) \ll  x^{2B^*_q+1}\log^2(2q) +x\log^2 (2q)\log^2x+\dfrac{x^3}{T^2}\log^4(xqT).$$
Now, we choose  $T$ tending to infinity. Therefore, since $B^*_q \geq \dfrac{1}{2}$ and $\log^2(2q)\leq q\leq x$, we assert that
$$H_{a,q}(x) \ll x^{2B^*_q+1}\log^2(qx).$$
\end{proof}

\begin{lem}\label{lemk}
For $1 \leq h \leq x $, we have 
$$K_{a,q}(x,h) \ll hx^{2B^*_q}\log ^2(qx).$$
\end{lem}
\begin{proof}

For $0<\theta\leq 1$, we consider 
\begin{align*}
K(x, \theta)&=\int_x^{2x} \left( \psi(t+t\theta,q,a)-\psi(t,q,a)-\dfrac{t\theta}{\varphi(q)}+\dfrac{\widetilde{\chi}(a)}{\varphi(q)}\dfrac{(t+t\theta)^{\widetilde{\beta}}-t^{\widetilde{\beta}}}{\widetilde{\beta}}\right) ^2dt\\
 &\ll \int_1^2 \int_{xv/2}^{2xv} \left( \psi(t+t\theta,q,a)-\psi(t,q,a)-\dfrac{t\theta}{\varphi(q)}+\dfrac{\widetilde{\chi}(a)}{\varphi(q)}\dfrac{(t+t\theta)^{\widetilde{\beta}}-t^{\widetilde{\beta}}}{\widetilde{\beta}}\right) ^2dt dv.
\end{align*}
Now, we use Lemma \ref{function_chi}, but we need to consider the case where $\chi$ is  imprimitive. Suppose that $\chi$ mod $q$ is induced by the primitive character $\chi^*$ mod $q^*$, $q>1$. Then, for $h\leq t$ we have
\begin{align*}
\psi(t+h,q,a)-\psi(t,q,a)=\dfrac{1}{\varphi(q)}\sum_{\chi(q)} \overline{\chi}(a)\left(\psi(t+h,\chi^*)-\psi(t,\chi^*)\right)+\mathcal{O}\left(\sum_{p|q}\sum_{t<p^k\le t+h}\log p\right),
\end{align*}
where 
\begin{align*}
    \sum_{p|q}\sum_{t<p^k\le t+h}\log p \ll \sum_{p|q}\left(\left[\dfrac{\log (t+h)}{\log p}\right]-\left[\dfrac{\log t}{\log p}\right]\right)\log p\ll \log q.
\end{align*}
\begin{align}\label{k1}
\begin{split}
&\int_{xv/2}^{2xv} \left( \psi(t+t\theta,q,a)-\psi(t,q,a)-\dfrac{t\theta}{\varphi(q)}+\dfrac{\widetilde{\chi}(a)}{\varphi(q)}\dfrac{(t+t\theta)^{\widetilde{\beta}}-t^{\widetilde{\beta}}}{\widetilde{\beta}}\right) ^2dt\\
&\ll\dfrac{1}{\varphi^2(q)} \int_{xv/2}^{2xv} \left| \sum_{\chi(q)}\overline{\chi}(a)\sum_{\substack{|\gamma_\chi|\leq T\\\rho_\chi\neq \widetilde{\beta}}}\dfrac{t^{\rho_\chi}}{\rho_\chi} ((1+\theta)^{\rho_\chi}-1)\right|^2 dt+\int_{xv/2}^{2xv}\left( \log (qt)+ \dfrac{t\log^2 (tqT)}{T} \right) ^2 dt\\
&=:K_1(v)+K_2(v).
\end{split}
\end{align}
Trivially, $K_2(v) \ll xv\log ^2(qx)+\dfrac{(xv)^3\log^4 (xqT)}{T^2} ,$ then
\begin{align}\label{k2}
\int_1^2 K_2(v) dv \ll x\log ^2(qx)+\dfrac{x^3\log^4 (xqT)}{T^2}.
\end{align}
By the Cauchy-Schwarz inequality, we obtain
\begin{align*}
K_1(v) &\ll \dfrac{1}{\varphi^2(q)} \sum_{\chi(q)}|\overline{\chi}(a)|^2\int_{xv/2}^{2xv} \left| \sum_{\substack{|\gamma_\chi|\leq T\\\rho_\chi\neq \widetilde{\beta}}}\dfrac{t^{\rho_\chi}}{\rho_\chi} ((1+\theta)^{\rho_\chi}-1)\right|^2 dt\\
& =\dfrac{1}{\varphi^2(q)}\sum_{\chi(q)}|\overline{\chi}(a)|^2\int_{xv/2}^{2xv} \sum_{\substack{|\gamma_\chi|\leq T\\\rho_\chi\neq \widetilde{\beta}}}\sum_{\substack{|\gamma'_\chi|\leq T\\\rho'_\chi\neq \widetilde{\beta}}}\dfrac{(1+\theta)^{\rho_\chi}-1}{\rho_\chi}.\dfrac{(1+\theta)^{\overline{\rho'_\chi}}-1}{\overline{\rho'_\chi}}t^{\rho_\chi} t^{\overline{\rho'_\chi}} dt\\
& = \dfrac{1}{\varphi^2(q)}\sum_{\chi(q)}|\overline{\chi}(a)|^2\sum_{\substack{|\gamma_\chi|\leq T\\\rho_\chi\neq \widetilde{\beta}}}\sum_{\substack{|\gamma'_\chi|\leq T\\\rho'_\chi\neq \widetilde{\beta}}}\dfrac{(1+\theta)^{\rho_\chi}-1}{\rho_\chi}.\dfrac{(1+\theta)^{\overline{\rho'_\chi}}-1}{\overline{\rho'_\chi}}\dfrac{2^{1+\rho_\chi+\overline{\rho'_\chi}}-2^{-1-\rho_\chi-\overline{\rho'_\chi}}}{\rho_\chi+\overline{\rho'_\chi}+1} (xv)^{\rho_\chi+\overline{\rho'_\chi}+1}.
\end{align*}
Thus, 
\begin{align*}
\int_1^2 K_1(v) dv \ll \dfrac{1}{\varphi^2(q)} \sum_{\chi(q)}|\overline{\chi}(a)|^2\sum_{\substack{|\gamma_\chi|\leq T\\\rho_\chi\neq \widetilde{\beta}}}\sum_{\substack{|\gamma'_\chi|\leq T\\\rho'_\chi\neq \widetilde{\beta}}}\dfrac{(1+\theta)^{\rho_\chi}-1}{\rho_\chi}.\dfrac{(1+\theta)^{\overline{\rho'_\chi}}-1}{\overline{\rho'_\chi}}&.\dfrac{2^{1+\rho_\chi+\overline{\rho'_\chi}}-2^{-1-\rho_\chi-\overline{\rho'_\chi}}}{\rho_\chi+\overline{\rho'_\chi}+1}\\
& .\dfrac{2^{2+\rho_\chi+\overline{\rho'_\chi}}-1}{\rho_\chi+\overline{\rho'_\chi}+2}x^{\rho_\chi+\overline{\rho'_\chi}+1}.
\end{align*}
By the trivial inequality $|z_1z_2| \ll  |z_1|^2+|z_2|^2$,

\begin{align*}
\int_1^2 K_1(v) dv &\ll \dfrac{1}{\varphi^2(q)} \sum_{\chi(q)}|\overline{\chi}(a)|^2\sum_{\substack{|\gamma_\chi|\leq T\\\rho_\chi\neq \widetilde{\beta}}}\sum_{\substack{|\gamma'_\chi|\leq T\\\rho'_\chi\neq \widetilde{\beta}}}\left|\dfrac{(1+\theta)^{\rho_\chi}-1}{\rho_\chi}.\dfrac{(1+\theta)^{\overline{\rho'_\chi}}-1}{\overline{\rho'_\chi}}\right| \dfrac{x^{\rho_\chi+\overline{\rho'_\chi}+1}}{\left|\rho_\chi+\overline{\rho'_\chi}+1\right|^2}\\
&\ll \dfrac{1}{\varphi^2(q)} x^{2B^*_q+1} \sum_{\chi(q)}|\overline{\chi}(a)|^2\sum_{\rho_\chi}\left|\dfrac{(1+\theta)^{\rho_\chi}-1}{\rho_\chi}\right|^2\sum_{\rho'_\chi}\dfrac{1}{\left|\rho_\chi+\overline{\rho'_\chi}+1\right|^2}\\
&\ll \dfrac{1}{\varphi^2(q)} x^{2B^*_q+1}\sum_{\chi(q)}|\overline{\chi}(a)|^2 \sum_{\gamma_\chi}\min \{\theta^2, (\gamma_\chi)^{-2}\}\sum_{\gamma'_\chi} \dfrac{1}{(1+|\gamma_\chi-\gamma'_\chi|)^2}. 
\end{align*}
By Lemma \ref{2zero}, we have
\begin{align*}
\sum_{\gamma'_\chi} \dfrac{1}{(1+|\gamma_\chi-\gamma'_\chi|)^2} \ll \log (q|\gamma_\chi|).
\end{align*}
From the above, 
\begin{align}\label{k3}
\begin{split}
\int_1^2 K_1(v) dv &\ll \dfrac{1}{\varphi^2(q)} x^{2B^*_q+1} \sum_{\chi(q)}|\overline{\chi}(a)|^2 \left( \sum_{\gamma_\chi> \theta^{-1}}\dfrac{\log (q\gamma_\chi)}{\gamma_\chi^2}+\sum_{0<\gamma_\chi\leq \theta^{-1}}\theta^2 \log (q\gamma_\chi)\right) \\
&\ll x^{2B^*_q+1}\theta\log^2\left( \dfrac{q}{\theta}\right),
\end{split}
\end{align} 
since $\sum_{\chi(q)}|\overline{\chi}(a)|^2=\varphi^2(q)$.  Substituting the estimates  \eqref{k2} and \eqref{k3} into \eqref{k1}, we then obtain
 \begin{align*}
 K(x, \theta) \ll x^{2B^*_q+1}\theta\log^2\left( \dfrac{q}{\theta}\right) +x\log ^2(qx)+\dfrac{x^3\log^4 (xqT)}{T^2}.
 \end{align*}
Similar to \cite[(6.21)]{SV}, we have
\begin{align*}
\int_x^{2x} &\left( \psi(t+h,q,a)-\psi(t,q,a)-\dfrac{h}{\varphi(q)}+\dfrac{\widetilde{\chi}(a)}{\varphi(q)}\dfrac{(t+h)^{\widetilde{\beta}}-t^{\widetilde{\beta}}}{\widetilde{\beta}}\right) ^2dt\\
&\ll \dfrac{x}{h}\int_{h/3x}^{3h/x}  \int_x^{3x}\left( \psi(t+t\theta,q,a)-\psi(t,q,a)-\dfrac{t\theta}{\varphi(q)}+\dfrac{\widetilde{\chi}(a)}{\varphi(q)}\dfrac{(t+t\theta)^{\widetilde{\beta}}-t^{\widetilde{\beta}}}{\widetilde{\beta}}\right) ^2dt d\theta\\
&\ll hx^{2B^*_q}\log^2\left( \dfrac{qx}{h}\right) +x\log ^2(qx)+\dfrac{x^3\log^4 (xqT)}{T^2}.
\end{align*}
We choose $T$ tending to infinity. Since $B^*_q\geq 1/2$, we obtain
\begin{align}\label{kx}
\int_x^{2x} &\left( \psi(t+h,q,a)-\psi(t,q,a)-\dfrac{h}{\varphi(q)}+\dfrac{\widetilde{\chi}(a)}{\varphi(q)}\dfrac{(t+h)^{\widetilde{\beta}}-t^{\widetilde{\beta}}}{\widetilde{\beta}}\right) ^2dt \ll h x^{2B^*_q} \log ^2 (qx).
\end{align}
Now we split $K_{a,q}(x,h)$ into two parts. 
\begin{align*}
K_{a,q}(x,h)&=\int_0^{h} \left( \psi(t+h,q,a)-\psi(t,q,a)-\dfrac{h}{\varphi(q)}+\dfrac{\widetilde{\chi}(a)}{\varphi(q)}\dfrac{(t+h)^{\widetilde{\beta}}-t^{\widetilde{\beta}}}{\widetilde{\beta}}\right) ^2dt\\
&+\int_h^{x} \left( \psi(t+h,q,a)-\psi(t,q,a)-\dfrac{h}{\varphi(q)}+\dfrac{\widetilde{\chi}(a)}{\varphi(q)}\dfrac{(t+h)^{\widetilde{\beta}}-t^{\widetilde{\beta}}}{\widetilde{\beta}}\right) ^2dt.
\end{align*}
The first is estimated as
\begin{align*}
&\ll \int_0^{h} \left( \psi(t+h,q,a)-\dfrac{t+h}{\varphi(q)}+\dfrac{\widetilde{\chi}(a)}{\varphi(q)}\dfrac{(t+h)^{\widetilde{\beta}}}{\widetilde{\beta}}\right) ^2dt+\int_0^{h}\left( \psi(t,q,a)-\dfrac{t}{\varphi(q)}+\dfrac{\widetilde{\chi}(a)}{\varphi(q)}\dfrac{t^{\widetilde{\beta}}}{\widetilde{\beta}}\right) ^2dt\\
&\ll h^{2B^*_q+1}\log^2(qx).
\end{align*}
We have $[h, x]$ is recovered by the disjoint union of $[\frac{x}{2^{k+1}},\frac{x}{2^k}]$, for $0 \leq k \leq \mathcal{O}(\log \dfrac{x}{h})$. Using \eqref{kx}, we rewrite the second part as
\begin{align*}
&\ll \sum_{k\geq 0}\int_{x/{2^{k+1}}}^{x/{2^{k}}}\left( \psi(t+h,q,a)-\psi(t,q,a)-\dfrac{h}{\varphi(q)}+\dfrac{\widetilde{\chi}(a)}{\varphi(q)}\dfrac{(t+h)^{\widetilde{\beta}}-t^{\widetilde{\beta}}}{\widetilde{\beta}}\right) ^2dt\\
&\ll h\sum_{k\geq 0}\left(\dfrac{x}{2^k}\right) ^{2B^*_q}\log^2(qx)\\
&\ll hx^{2B^*_q}\log^2(qx).
\end{align*}
Combining two estimate above, the proof is complete.
\end{proof}

\section{Proof of the part (A) Theorem \ref{t1} }
From Proposition \ref{t2}, to prove (A) Theorem \ref{t1}, it is enough to estimate $$E(X,q_1,q_2,a_1,a_2)\ll X^{B^*_{q_1}+B^*_{q_2}}\log X\log (q_1X) \log (q_2X).$$
By the Cauchy-Schwarz inequality, the integral $E(X,q_1,q_2,a_1,a_2)$ is estimated as

\begin{align*}
&|E(X,q_1,q_2,a_1,a_2)|^2=\left|\int^1_0 \left(  F_{a_1,q_1}(z)-I_{q_1}(z)+\widetilde{I}_{q_1}(z)\right) \left(  F_{a_2,q_2}(z)-I_{q_2}(z)+\widetilde{I}_{q_2}(z)\right)  I\left(X, \dfrac{1}{z}\right)  d\alpha. \right|^2\\
&\ll  \int^1_0 \left| F_{a_1,q_1}(z)-I_{q_1}(z)+\widetilde{I}_{q_1}(z)\right|^2 \left|I\left(X, \dfrac{1}{z}\right) \right| d\alpha \int^1_0 \left|F_{a_2,q_2}(z)-I_{q_2}(z)+\widetilde{I}_{q_2}(z)\right| ^2\left|I\left(X, \dfrac{1}{z}\right) \right| d\alpha.
\end{align*}
We remark that $|E(X,q_1,q_2,a_1,a_2)|$ can be estimated as the product of two factors which have the same properties. 
 So we need only consider 
\begin{align}
T(X,q,a):=\int^1_0 \left|F_{a,q}(z)-I_q(z)+\widetilde{I}_{q}(z)\right|^2 \left|I\left(X, \dfrac{1}{z}\right) \right| d\alpha,
\end{align}
we remark that 
 $$|E(X, q_1, q_2, a_1, a_2))|\ll \max_{\substack{q\in \{q_1,q_2\}\\ a\in \{a_1,a_2\}}}
(T(X,q,a)).$$

We will estimate $T(X,q,a)$ by bounding  $I\left(X, \dfrac{1}{z}\right)$ and then the moment of $\left|F_{a,q}(z)-I_q(z)+\widetilde{I}_q(z)\right|$. We first remark that , for $1 \leq y \leq X$, by  partial summation,
  $$I\left(y, \dfrac{1}{z}\right)=\sum_{n\leq y}e^{n/X}e(-n\alpha)=e^{y/X}\sum_{n\leq y}e(-n\alpha)  - \dfrac{1}{X} \int_1^y \sum_{n\leq t}e(-n\alpha)e^{t/X} dt.$$
  Since $$\sum_{n\leq t}e(-n\alpha) \ll \min \left( t, \dfrac{1}{\left|\alpha\right|}\right) ,$$
 we obtain 
  \begin{align}\label{min}
  I\left(y, \dfrac{1}{z}\right) \ll \min \left( y, \dfrac{1}{\left|\alpha\right|}\right).
  \end{align}
We now prove the following lemma.
\begin{lem}\label{lem1}
For $X \geq 2$, $\left|\alpha\right| \leq 1/2$  and an  integrable positive function $f$ of period $1$, we have 

\begin{align}\label{I}
\int^{1/2}_{-1/2} f(\alpha) \min \left( X, \dfrac{1}{\left|\alpha\right|}\right)  d\alpha \ll \sum_{k=0}^ {\mathcal{O}(\log X)}\frac{X}{2^k}\int_0^{\min(2^{k+1}/X,1/2)}f({\alpha}) d\alpha.
\end{align}
\end{lem}

\begin{proof} Because the integrand is an even function., we can restrict ourselves to $[0, 1/2]$. We evaluate the left-hand side of \eqref{I} 
\begin{align*}
\int^{1/2}_0 f(\alpha) \min \left( X, \dfrac{1}{\left|\alpha\right|}\right)  d\alpha=X\int^{1/X}_0 f(\alpha) d\alpha + \int^{1/2}_{1/X} f(\alpha) \dfrac{d\alpha}{\alpha}.
\end{align*}
We write $[\frac{1}{X}, \frac{1}{2}]$ as the disjoint union of $[\frac{2^{k}}{X},\frac{2^{k+1}}{X}]$, $0\leq k\leq \mathcal{O}(\log X)$. Then we rewrite the right-hand side of the above formula as 
\begin{align*}
&X\int^{1/X}_0 f(\alpha) d\alpha +\sum_{k=0}^{\mathcal{O}(\log X)}\int^{\min(2^{k+1}/X,1/2)}_{2^{k}/X} f(\alpha) \dfrac{d\alpha}{\alpha}\\
&\leq  X\int^{1/X}_0 f(\alpha) d\alpha+ \sum_{k=0}^{\mathcal{O}(\log X)}\dfrac{X}{2^k}\int^{\min(2^{k+1}/X,1/2)}_{2^{k}/X} f(\alpha) d\alpha\\
&\ll \sum_{k=0}^{\mathcal{O}(\log X)}\dfrac{X}{2^k}\int_0^{\min(2^{k+1}/X,1/2)}f({\alpha}) d\alpha.
\end{align*}
\end{proof}
Since $z=z(\alpha)=e^{-1/X}e(n\alpha)$, we replace $f(\alpha)=\left|F_{a,q}(z)-I_q(z)+\widetilde{I}_{q}(z)\right|^2$ in Lemma \ref{lem1}, to obtain
\begin{align}\label{T}
T(X,q,a) &\ll \sum_{k=0}^{\mathcal{O}(\log X)}\frac{X}{2^k}\int_0^{\min(2^{k+1}/X,1/2)}\left|F_{a,q}(z)-I_q(z)+\widetilde{I}_{q}(z)\right|^2 d\alpha.
\end{align}
For $1\leq h\leq X$, putting
 \begin{align*}
W(X,q,a,h):&=\int_0^{1/2h}\left|F_{a,q}(z)-I_q(z)+\widetilde{I}_{q}(z)\right|^2 d\alpha\\
&=\int_0^{1/2h}\left| \sum_n \left( \Lambda(n)\delta(n)-\dfrac{1}{\varphi(q)}+\dfrac{\widetilde{\chi}(a)}{\varphi(q)}n^{\widetilde{\beta}-1}\right) e^{-n/X} e(n\alpha)\right|^2 d\alpha,
\end{align*}
where $\delta(n)$ is defined as
$$\delta(n)=1 \text{ if }n\equiv a (\text{mod }q) \text{ and }0 \text{ otherwise}.$$
Using  \cite[Gallagher's lemma]{MV}, for $1\leq h\leq X$, we obtain
\begin{align}\label{W}
\begin{split}
W(X,q,a,h)&\ll \dfrac{1}{h^2 }\int_{-h}^\infty \left| \sum_{x<n\leq x+h} \left( \Lambda(n)\delta(n)-\dfrac{1}{\varphi(q)}+\dfrac{\widetilde{\chi}(a)}{\varphi(q)}n^{\widetilde{\beta}-1}\right) e^{-n/X}\right|^2dx\\
&\ll  \dfrac{1}{h^2 }\int_{0}^h \left| \sum_{n\leq x} \left( \Lambda(n)\delta(n)-\dfrac{1}{\varphi(q)}+\dfrac{\widetilde{\chi}(a)}{\varphi(q)}n^{\widetilde{\beta}-1}\right) e^{-n/X}\right|^2dx\\
&+\dfrac{1}{h^2 }\int_{0}^\infty \left| \sum_{x<n\leq x+h} \left( \Lambda(n)\delta(n)-\dfrac{1}{\varphi(q)}+\dfrac{\widetilde{\chi}(a)}{\varphi(q)}n^{\widetilde{\beta}-1}\right) e^{-n/X}\right|^2dx.
\end{split}
\end{align}
Next we evaluate 
\begin{align*}
&J_{1,a,q}(X,h):=\int_{0}^h \left| \sum_{n\leq x} \left( \Lambda(n)\delta(n)-\dfrac{1}{\varphi(q)}+\dfrac{\widetilde{\chi}(a)}{\varphi(q)}n^{\widetilde{\beta}-1}\right) e^{-n/X}\right|^2dx\\
&J_{2,a,q}(X,h):=\int_{0}^\infty \left| \sum_{x<n\leq x+h} \left( \Lambda(n)\delta(n)-\dfrac{1}{\varphi(q)}+\dfrac{\widetilde{\chi}(a)}{\varphi(q)}n^{\widetilde{\beta}-1}\right) e^{-n/X}\right|^2dx
\end{align*} from the two following lemmas.
\begin{lem} \label{lemJ1}
For $X \geq 2$ and $1\leq h\leq X$, we have
\begin{align*}
J_{1,a,q}(X,h) \ll  H_{a,q}(h)+\dfrac{h}{\varphi^2(q)}.
\end{align*}
\end{lem}

\begin{proof}
Let us write   $$F(n):= \Lambda(n)\delta(n)-\dfrac{1}{\varphi(q)}+\dfrac{\widetilde{\chi}(a)}{\varphi(q)}n^{\widetilde{\beta}-1}.$$ For $y \geq 2$, using partial summation
\begin{align}\label{summation by parts}
\sum_{n\leq y}F(n) e^{-n/X}=e^{-y/X}\sum_{n\leq y}F(n)+\dfrac{1}{X}\int_0^y \sum_{n\leq t}F(n) e^{-t/X} dt,
\end{align}
we then have 
\begin{align*}
    J_{1,a,q}(X,h)=\int_{0}^h \left| \sum_{n\leq x} F(n) e^{-n/X}\right|^2dx \ll \int_{0}^h \left| \sum_{n\leq x} F(n)\right|^2dx,\\
\end{align*}
since by \eqref{summation by parts}, we have
\begin{align*}
    \int_{0}^h \left| \sum_{n\leq x} F(n) e^{-n/X}\right|^2dx&\ll e^{-2h/X} \int_{0}^h \left| \sum_{n\leq x} F(n)\right|^2dx+ \dfrac{1}{X^2}\int_0^h\left(\int_0^x\left| \sum_{n\leq t} F(n)\right|^2dt\int_0^xe^{-2t/X}dt\right)dx\\
    &\ll \int_{0}^h \left| \sum_{n\leq x} F(n)\right|^2dx+ \dfrac{h^2}{X^2}\int_{0}^h \left| \sum_{n\leq t} F(n)\right|^2dt \ll \int_{0}^h \left| \sum_{n\leq x} F(n)\right|^2dx.
\end{align*}
Hence, by \eqref{sum-integral}, one has
\begin{align}
\begin{split}
J_{1,a,q}(X,h)&\ll \int_{0}^h  \left(\psi(x,q,a)-\dfrac{[x]}{\varphi(q)}+\dfrac{\widetilde{\chi}(a)}{\varphi(q)}\sum_{n \leq x}n^{\widetilde{\beta}-1}\right)^2 dx\\
& \ll \int_{0}^h  \left(\psi(x,q,a)-\dfrac{\widetilde{\chi}(a)}{\varphi(q)}\dfrac{x^{\widetilde{\beta}}}{\widetilde{\beta}}\right)^2 dx+\dfrac{1}{\varphi(q)^2}\int_{0}^h (x-[x])^2dx+ \dfrac{h}{\varphi^2(q)}\\
&\leq H_{a,q}(h)+\dfrac{h}{\varphi^2(q)}.
\end{split}
\end{align}
\end{proof}

\begin{lem} \label{lemJ2}
For $X \geq 2$ and $1\leq h\leq X$, we have
\begin{align*}
J_{2,a,q}(X,h) \ll \dfrac{h^2}{X^2}\sum_{j=1}^\infty \dfrac{1}{2^{j-1}} \left( H_{a,q}(jX)+\dfrac{jX}{\varphi^2(q)}\right)+\sum_{j=1}^\infty \dfrac{1}{2^{j-1}} \left( K_{a,q}(jX,h)+\dfrac{jX}{\varphi^2(q)}\right).
\end{align*}
\end{lem}

\begin{proof}
We replace $y$ in \eqref{summation by parts}
by $x+h$ and $x$,  we obtain

\begin{align}\label{f}
\begin{split}
\sum_{x<n\leq x+h}&F(n) e^{-n/X}= \left( \sum_{n\leq x+h}-\sum_{n\leq x}\right) F(n) e^{-n/X}\\
&=e^{-(x+h)/X}\sum_{n\leq x+h}F(n)+\dfrac{1}{X}\int_0^{x+h} \sum_{n\leq t}F(n) e^{-t/X} dt-e^{-x/X}\sum_{n\leq x}F(n)-\dfrac{1}{X}\int_0^x \sum_{n\leq t}F(n) e^{-t/X} dt\\
&=e^{-x/X}\sum_{x<n\leq x+h}F(n)+e^{-x/X}\sum_{n\leq x+h}F(n)(e^{-h/X}-1)+\dfrac{1}{X}\int_x^{x+h} \sum_{n\leq t}F(n) e^{-t/X} dt.
\end{split}
\end{align}

Since $e^{-t}-1 \leq t$ for all $t$ non negative, then $e^{-h/X}-1 \leq h/X$. Now squaring  and integrating  \eqref{f}, we obtain 
\begin{align}\label{F}
\begin{split}
J_{2,a,q}(X,h)&=\int_{0}^\infty \left| \sum_{x<n\leq x+h}F(x) e^{-n/X}\right|^2dx\\
&\ll \int_{0}^\infty \left( \sum_{x<n\leq x+h}F(n)\right)^2 e^{-2x/X} dx+\dfrac{h^2}{X^2}\int_{0}^\infty \left( \sum_{n\leq x+h}F(n)\right)^2 e^{-2x/X}dx\\
&+ \dfrac{1}{X^2} \int_{0}^\infty \left( \int_x^{x+h} \sum_{n\leq t}F(n) e^{-t/X} dt\right) ^2 dx.
\end{split}
\end{align}
The first term on the right hand-side of \eqref{F} is
\begin{align*}
&\int_{0}^\infty \left( \sum_{x<n\leq x+h}F(n)\right)^2 e^{-2x/X}dx= \sum_{j=0}^\infty\int_{jX}^{(j+1)X}\left( \sum_{x<n\leq x+h}F(n)\right)^2 e^{-2x/X}dx\\
&\leq \sum_{j=0}^\infty e^{-2j}\int_{jX}^{(j+1)X} \left( \sum_{x<n\leq x+h}F(n)\right)^2dx
\leq \sum_{j=0}^\infty \dfrac{1}{2^j}\int_{0}^{(j+1)X} \left( \sum_{x<n\leq x+h}F(n)\right)^2dx,
\end{align*}
since $e^{-2j}\leq 2^{-j}$ for all $j$ non-negative integers. Moreover, by \eqref{sum-integral}, we have
\begin{align*}
\int_{0}^y  \left( \sum_{x<n\leq x+h}F(n)\right)^2dx&\ll \int_0^y \left( \psi(x+h,q,a)-\psi(x,q,a)-\dfrac{[x+h]}{\varphi(q)}+\dfrac{[x]}{\varphi(q)}+ \dfrac{\widetilde{\chi}(a)}{\varphi(q)}\sum_{x<n\leq x+h}n^{\widetilde{\beta}-1}\right) ^2dx\\
&\ll \int_0^y \left( \psi(x+h,q,a)-\psi(x,q,a)-\dfrac{h}{\varphi(q)}+ \dfrac{\widetilde{\chi}(a)}{\varphi(q)}\dfrac{(x+h)^{\widetilde{\beta}}-x^{\widetilde{\beta}}}{\widetilde{\beta}}\right) ^2dx\\
&+\int_0^y \left( \dfrac{[x+h]}{\varphi(q)}-\dfrac{[x]}{\varphi(q)}-\dfrac{h}{\varphi(q)}\right) ^2dx+\dfrac{y}{\varphi^2(q)}\\
&\ll K_{a,q}(y,h)+\dfrac{y}{\varphi^2(q)}.
\end{align*} 
Hence,
\begin{align}\label{p1}
\int_{0}^\infty \left( \sum_{x<n\leq x+h}F(n)\right)^2 e^{-2x/X}dx  \ll \sum_{j=1}^\infty \dfrac{1}{2^{j-1}} \left( K_{a,q}(jX,h)+\dfrac{jX}{\varphi^2(q)}\right).
\end{align}
Next we consider the last term of \eqref{F}. By the Cauchy-Schwarz inequality, 
\begin{align*}
 \int_{0}^\infty \left( \int_x^{x+h} \sum_{n\leq t}F(n) e^{-t/X} dt\right) ^2 dx &\ll h\int_{0}^\infty   \int_x^{x+h} \left( \sum_{n\leq t}F(n)\right) ^2 e^{-2t/X} dt dx\\
 &=h\int_{0}^\infty \left( \sum_{n\leq t}F(n)\right) ^2 e^{-2t/X} \left( \int_{t-h}^{t} dx\right)   dt\\
 &=h^2\int_{0}^\infty  \left( \sum_{n\leq t}F(n)\right) ^2 e^{-2t/X} dt.
\end{align*}
 Moreover, similar to the estimate of the first term, we have
\begin{align*}
\int_{0}^\infty \left( \sum_{n\leq t}F(n)\right)^2 e^{-2t/X}dt &\leq \sum_{j=0}^\infty \dfrac{1}{2^j}\int_{0}^{(j+1)X} \left( \sum_{n\leq t}F(n)\right)^2dt\ll \sum_{j=1}^\infty \dfrac{1}{2^{j-1}} \left( H_{a,q}(jX)+\dfrac{jX}{\varphi^2(q)}\right) .
\end{align*}
Then the last term of \eqref{F} is 
\begin{align}\label{p3}
\ll \dfrac{h^2}{X^2}\sum_{j=1}^\infty \dfrac{1}{2^{j-1}} \left( H_{a,q}(jX)+\dfrac{jX}{\varphi^2(q)}\right).
\end{align}
Finally, the second term of \eqref{F} is 
\begin{align}\label{p2}
\dfrac{h^2}{X^2}\int_{0}^\infty \left( \sum_{n\leq x+h}F(n)\right)^2 e^{-2x/X}dx=e^{2h/X}\dfrac{h^2}{X^2}\int_{h}^\infty \left( \sum_{n\leq t}F(n)\right)^2 e^{-2t/X}dt 
\end{align}
\begin{align*}
\hspace{6cm} \ll \dfrac{h^2}{X^2}\sum_{j=1}^\infty \dfrac{1}{2^{j-1}} \left( H_{a,q}(jX)+\dfrac{jX}{\varphi^2(q)}\right).
\end{align*}
Substituting \eqref{p1}, \eqref{p3}, \eqref{p2} into  \eqref{F}, the proof is complete.
\end{proof}

Now, using the estimates in Lemmas  \ref{lemh} and \ref{lemk}, we have 
\begin{align*}
    J_{1,a,q}(X,h)& \ll  H_{a,q}(h)+\dfrac{h}{\varphi^2(q)}\ll h^{2B_q^*+1}\log^2(qh)+\dfrac{h}{\varphi^2(q)}\\
    J_{2,a,q}(X,h)& \ll \dfrac{h^2}{X^2} \sum_{j=1}^\infty \dfrac{1}{2^{j-1}} \left( (jX)^{2B_q^*+1}\log^2(qjX) +\dfrac{jX}{\varphi^2(q)}\right)+\sum_{j=1}^\infty \dfrac{1}{2^{j-1}} \left( h(jX)^{2B_q^*} \log^2 (qjX)+\dfrac{jX}{\varphi^2(q)}\right)\\
    &\ll h^2X^{2B^*_q-1}\log^2(qX)\sum_{j \geq 1}\dfrac{j^{2B^*_q+1}}{2^{j-1}}+hX^{2B^*_q}\log ^2 (qX)\sum_{j \geq 1}\dfrac{j^{2B^*_q}}{2^{j-1}} \\
    &\ll hX^{2B^*_q}\log^2(qX).
\end{align*}
By \eqref{W}, we can infer
\begin{align*}\label{ww}
W(X,q,a,h)  \ll& \dfrac{J_{1,a,q}(X,h)}{h^2}+\dfrac{J_{2,a,q}(X,h)}{h^2} \\
 \ll&\dfrac{X^{2B^*_q}}{h} \log^2(qX).
\end{align*}
Substituting the above estimate into \eqref{T}, we get
\begin{align*}
T(X,q,a) &\ll X^{2B^*_q} \log X \log^2(qX).
\end{align*}
We conclude 
\begin{align*}
E(X,q_1,q_2,a_1,a_2) &\ll T(X,q_1,a_1)^{1/2}T(X,q_2,a_2)^{1/2} \\
&\ll X^{B^*_{q_1}+B^*_{q_2}} \log X \log(q_1X)\log(q_2X).
\end{align*}
Hence we complete the proof of part (A) Theorem \ref{t1}. 

\section{Proof of the part (B) Theorem \ref{t1}}
By the part (A), we obtain the main term in this formula. We will show that 
\begin{equation}\label{G}
G(n,q_1,q_2,a_1,a_2)=\Omega (n\log \log n),
\end{equation}
 and then $(B)$ is proved. In fact that, assuming the error term of $(B)$ is $o(X\log \log X)$ then
\begin{align*}
&S(n+1,q_1,q_2,a_1,a_2)-S(n,q_1,q_2,a_1,a_2)=\dfrac{(n+1)^2-n^2}{2\varphi(q_1)\varphi(q_2)}+ \dfrac{1}{\varphi(q_2)} \left( H(n+1,q_1,a_1)-H(n,q_1,a_1)\right)\\
&+\dfrac{1}{\varphi(q_1)} \left( H(n+1,q_2,a_2)-H(n,q_2,a_2)\right) +\mathcal{Z}(n+1, \widetilde{\beta}_1, \widetilde{\beta}_2)-\mathcal{Z}(n, \widetilde{\beta}_1, \widetilde{\beta}_2) + o(n\log \log n)\\
&=o(n\log \log n).
\end{align*}
So \eqref{G} would be false. Hence the error term of $(B)$ is $\Omega(X\log \log X)$.

We recall the following result.
\begin{lem}\label{GP} 
Let $q'$ be a positive integer,  $q=\prod_{\substack{p \leq x \\p\neq p_1, p\nmid q'}}p$, where $p_1$ is some prime divisor of the exceptional modulus up to $q$ if there exists a Siegel's zero, then
\begin{align*}
\left| x-\sum_{x\leq n\leq 2x}\Lambda(n)\chi_0(n) \right| + \sum_{\substack{\chi(q)\\\chi \neq \chi_0}} \left| \sum_{x\leq n\leq 2x}\Lambda(n)\chi(n) \right| \leq \dfrac{x}{2}.
\end{align*}
\end{lem}
This lemma is implied by \cite[Theorem 7]{Gallagher}, we also see it in \cite[Lemma 4]{BP}. It follows that, for $(a,q)=1$,
\begin{align*}
\psi(2x,q,a)&=\sum_{\substack{n\leq 2x\\n\equiv a(q)}}\Lambda(n) \geq \sum_{\substack{x \leq n\leq 2x\\n\equiv a(q)}}\Lambda(n)=\dfrac{1}{\varphi(q)}\sum_{\chi(q)}\overline{\chi}(a) \sum_{x\leq n \leq 2x}\chi(n)\Lambda(n)\\
&=\dfrac{1}{\varphi(q)}\sum_{x\leq n\leq 2x}\Lambda(n)\chi_0(n) +\dfrac{1}{\varphi(q)}\sum_{\substack{\chi(q)\\\chi \neq \chi_0}}\overline{\chi}(a) \sum_{x\leq n \leq 2x}\chi(n)\Lambda(n)\\
&=\dfrac{x}{\varphi(q)}-\dfrac{1}{\varphi(q)}\left(x-\sum_{x\leq n\leq 2x}\Lambda(n)\chi_0(n)-\sum_{\substack{\chi(q)\\\chi \neq \chi_0}} \overline{\chi}(a) \sum_{x\leq n\leq 2x}\Lambda(n)\chi(n) \right)\\
&\geq \dfrac{x}{\varphi(q)}-\dfrac{1}{\varphi(q)}\left(\left|x-\sum_{x\leq n\leq 2x}\Lambda(n)\chi_0(n)\right|+\sum_{\substack{\chi(q)\\\chi \neq \chi_0}} \left|\overline{\chi}(a)\right|\left| \sum_{x\leq n\leq 2x}\Lambda(n)\chi(n)\right| \right)\\
& \geq \dfrac{x}{\varphi(q)}-\dfrac{x}{2\varphi(q)}=\dfrac{x}{2\varphi(q)}.
\end{align*}
Let us now $$Q=\prod_{\substack{p \leq x \\p\neq p_1, p\nmid q_1q_2}}p.$$ Hence, 
\begin{align}\label{Q}
\sum_{\substack{n \leq 4x\\Q\mid n}}G(n,q_1,q_2,a_1,a_2)\geq \sum_{\substack{b=1\\(b,Q)=1}}^Q\sum_{\substack{m\leq 2x\\m\equiv a_1(q_1)\\m\equiv b(Q)}}\Lambda(m) \sum_{\substack{l\leq 2x\\l\equiv a_2(q_2)\\l\equiv Q-b(Q)}}\Lambda(l).
\end{align}
We use the Chinese remainder Theorem to evaluate the right-hand side of \eqref{Q}. Since $Q$ and $q_1$ are coprime, the solutions of the system $m\equiv a_1(q_1)$ and $m\equiv b(Q)$ are given by
$$m\equiv a_1m_1Q+bm_2q_1(q_1Q):=A (q_1Q),$$
where $m_1, m_2$ are integers satisfying $m_1Q+m_2q_1=1$.

Similar to the system $l\equiv a_2(q_2)$ and $l\equiv Q-b(Q)$, we have
$$l=:B(q_2Q).$$
Thus, using Lemma \ref{GP},  the right-hand side of \eqref{Q} is 
\begin{align*}
&\sum_{\substack{b=1\\(b,Q)=1}}^Q\sum_{\substack{m\leq 2x\\m\equiv A(q_1Q)}}\Lambda(m) \sum_{\substack{l\leq 2x\\l\equiv B(q_2Q)}}\Lambda(l) =\sum_{\substack{b=1\\(b,Q)=1}}^Q\psi(2x,q_1Q,A)\psi(2x,q_2Q,B)\\
&\geq \sum_{\substack{b=1\\(b,Q)=1}}^Q \dfrac{x}{2\varphi(q_1Q)}\dfrac{x}{2\varphi(q_2Q)}=\dfrac{x^2}{4\varphi(q_1)\varphi(q_2)\varphi(Q)},
\end{align*}
since $(q_1,Q)=1$, $(q_2,Q)=1$.

Hence, we obtain 
\begin{align*}
\max_{n\leq 4x}G(n,q_1,q_2,a_1,a_2) &\geq \dfrac{Q}{4x}\sum_{\substack{n \leq 4x\\Q\mid n}}G(n,q_1,q_2,a_1,a_2) \\
&\geq \dfrac{xQ}{16\varphi(q_1)\varphi(q_2)\varphi(Q)}=\dfrac{x}{16\varphi(q_1)\varphi(q_2)} \prod_{\substack{p \leq x \\p\neq p_1, p\nmid q_1q_2}}(1-p^{-1})^{-1}\\
&=(1-p_1^{-1})\dfrac{\prod_{ p\mid q_1q_2}(1-p^{-1})}{16\varphi(q_1)\varphi(q_2)}x\prod_{p \leq x }(1-p^{-1})^{-1}\\
&\geq (1-p_1^{-1})\dfrac{\prod_{ p\mid q_1}(1-p^{-1})\prod_{ p\mid q_2}(1-p^{-1})}{16\varphi(q_1)\varphi(q_2)}x\prod_{p \leq x }(1-p^{-1})^{-1} \\
&\gg_{q_1,q_2}
x\log \log x.
\end{align*}
\section*{Acknowledgement} 
I would like to thank Didier Lesesvre for inspiring discussions on the topics of this paper. I would also like to thank Alessandro Languasco for the his careful reading and his comments. I am grateful to Gautami Bhowmik for her guidance together with many useful comments and suggestions during the preparation of this paper. I would like to express my sincere thanks to Yuta Suzuki for his suggestions, which helped me improve my results.
 \bibliographystyle{amsplain}
\bibliography{anbib}

\providecommand{\bysame}{\leavevmode\hbox to3em{\hrulefill}\thinspace}
\providecommand{\MR}{\relax\ifhmode\unskip\space\fi MR }
\providecommand{\MRhref}[2]{%
  \href{http://www.ams.org/mathscinet-getitem?mr=#1}{#2}
}
\providecommand{\href}[2]{#2}
\begin{thebibliography}{10}

\bibitem{BHMS}
Gautami Bhowmik, Karin Halupczok, Kohji Matsumoto, and Yuta Suzuki,
  \emph{Goldbach representations in arithmetic progressions and zeros of
  {D}irichlet {$L$}-functions}, Mathematika \textbf{65} (2019), no.~1, 57--97.
  \MR{3867327}

\bibitem{BP}
Gautami Bhowmik and Jan-Christoph Schlage-Puchta, \emph{Mean representation
  number of integers as the sum of primes}, Nagoya Math. J. \textbf{200}
  (2010), 27--33. \MR{2747876}

\bibitem{Egami-Mat}
Shigeki Egami and Kohji Matsumoto, \emph{Convolutions of the von {M}angoldt
  function and related {D}irichlet series}, Number theory, Ser. Number Theory
  Appl., vol.~2, World Sci. Publ., Hackensack, NJ, 2007, pp.~1--23.
  \MR{2364835}

\bibitem{Fujii}
Akio Fujii, \emph{An additive problem of prime numbers}, Acta Arith.
  \textbf{58} (1991), no.~2, 173--179. \MR{1121079}

\bibitem{Gallagher}
P.~X. Gallagher, \emph{A large sieve density estimate near {$\sigma =1$}},
  Invent. Math. \textbf{11} (1970), 329--339. \MR{279049}

\bibitem{GS}
Daniel~A. Goldston and Ade~Irma Suriajaya, \emph{On an average {G}oldbach
  representation formula of {F}ujii}, Nagoya Math. J. \textbf{250} (2023),
  511--532. \MR{4583139}

\bibitem{GY}
Daniel~A. Goldston and Liyang Yang, \emph{The average number of goldbach
  representations}, in Prime Nos and Rep. In, Vol2 (2017).

\bibitem{granville1}
Andrew Granville, \emph{Refinements of {G}oldbach's conjecture, and the
  generalized {R}iemann hypothesis}, Funct. Approx. Comment. Math. \textbf{37}
  (2007), 159--173. \MR{2357316}

\bibitem{granville2}
\bysame, \emph{Corrigendum to ``{R}efinements of {G}oldbach's conjecture, and
  the generalized {R}iemann hypothesis'' [mr2357316]}, Funct. Approx. Comment.
  Math. \textbf{38} (2008), 235--237. \MR{2492859}

\bibitem{HL}
G.~H. Hardy and J.~E. Littlewood, \emph{Some problems of `{P}artitio
  numerorum'; {III}: {O}n the expression of a number as a sum of primes}, Acta
  Math. \textbf{44} (1923), no.~1, 1--70. \MR{1555183}

\bibitem{Lang-Zacca1}
Alessandro Languasco and Alessandro Zaccagnini, \emph{The number of {G}oldbach
  representations of an integer}, Proc. Amer. Math. Soc. \textbf{140} (2012),
  no.~3, 795--804. \MR{2869064}

\bibitem{MV}
Hugh~L. Montgomery and Robert~C. Vaughan, \emph{Multiplicative number theory.
  {I}. {C}lassical theory}, Cambridge Studies in Advanced Mathematics, vol.~97,
  Cambridge University Press, Cambridge, 2007. \MR{2378655}

\bibitem{ruppel2012}
Frederike R\"{u}ppel, \emph{Convolutions of the von {M}angoldt function over
  residue classes}, \v{S}iauliai Math. Semin. \textbf{7(15)} (2012), 135--156.
  \MR{2997066}

\bibitem{SV}
B.~Saffari and R.~C. Vaughan, \emph{On the fractional parts of {$x/n$} and
  related sequences. {II}}, Ann. Inst. Fourier (Grenoble) \textbf{27} (1977),
  no.~2, v, 1--30. \MR{480388}

\bibitem{YS}
Yuta Suzuki, \emph{A mean value of the representation function for the sum of
  two primes in arithmetic progressions}, Int. J. Number Theory \textbf{13}
  (2017), no.~4, 977--990. \MR{3627693}

\end{thebibliography}

\end{document}